\documentclass[9pt,a4paper,twoside]{amsart}
\usepackage{amsmath,amssymb,times,textcomp,amsthm}

\usepackage[utf8]{inputenc}
\usepackage[T1]{fontenc}
\usepackage[all]{xy}
\usepackage[pdftex]{graphicx}
\usepackage[hypertexnames=false,backref=page,pdftex,
 	pdfpagemode=UseNone,
 	breaklinks=true,
 	extension=pdf,
 	colorlinks=true,
 	linkcolor=blue,
 	citecolor=red,
 	urlcolor=blue,
 ]{hyperref}
\usepackage{color}
\usepackage[english]{babel}

\newtheoremstyle{ourthm}% <name>
{10pt}% <Space above>
{3pt}% <Space below>
{\slshape}% <Body font>
{}% <Indent amount>
{\bf\upshape}% <Theorem head font>
{}% <Punctuation after theorem head>
{.5em}% <Space after theorem headi>
{}% <Theorem head spec (can be left empty, meaning `normal')> 

\theoremstyle{ourthm}
\newtheorem{theorem}{Theorem}[section]

\newtheorem{lemma}[theorem]{Lemma}

\newtheorem{proposition}[theorem]{Proposition}
\newtheorem{corollary}[theorem]{Corollary}

\theoremstyle{definition}

\newtheorem{example}[theorem]{Example}
\newtheorem{conjecture}[theorem]{Conjecture}

\newtheorem{remark}[theorem]{Remark}

\newenvironment{beweis*}[1]{\noindent{\sl #1}}{\qed}

\newcommand{\ko}{{\mathcal O}}

\newcommand{\IC}{{\mathbb C}}

\newcommand{\IG}{{\mathbb G}}

\newcommand{\IN}{{\mathbb N}}

\newcommand{\IZ}{{\mathbb Z}}

\newcommand{\gothc}{{\mathfrak c}}

\newcommand{\gothg}{{\mathfrak g}}
\newcommand{\gothh}{{\mathfrak h}}

\newcommand{\gothk}{{\mathfrak k}}
\newcommand{\gothl}{{\mathfrak l}}
\newcommand{\gothm}{{\mathfrak m}}

\newcommand{\gotho}{{\mathfrak o}}

\newcommand{\goths}{{\mathfrak s}}

\newcommand{\beeq}[1]{\begin{eqnarray}\label{#1}}
\newcommand{\eneq}{\end{eqnarray}}

\newcommand{\dual}{\makebox[0mm]{}^{{\scriptstyle\vee}}}
\newcommand{\scrdual}{\makebox[0mm]{}^{{\scriptscriptstyle\vee}}}

\newcommand{\isom}{\cong}

\newcommand{\tensor}{\otimes}

\newcommand{\Spec}{{\rm Spec}}

\newcommand{\Hom}{{\rm Hom}}

\newcommand{\id}{{\rm id}}

\newcommand{\rk}{{\rm rk}}

\newcommand{\reg}{{\rm reg}}

\DeclareMathOperator{\LieSl}{SL}
\DeclareMathOperator{\LieSp}{Sp}
\DeclareMathOperator{\LieSpin}{Spin}

\DeclareMathOperator{\LieGl}{GL}

\newcommand{\Liegl}{\gothg\gothl}

\newcommand{\Lieso}{{\goths\gotho}}

\newcommand{\Liesl}{{\goths\gothl}}

\newcommand{\Kern}{{\rm ker}}

\renewcommand{\mod}{{\rm mod}}

\newcommand{\pr}{{\rm pr}}

\newcommand{\lra}{\longrightarrow}
\newcommand{\xra}{\xrightarrow}

\newcommand{\verylongarrow}[1]{\raisebox{0mm}[0.9ex][0ex]{\hbox to #1{\rightarrowfill}}}

\DeclareMathOperator{\tr}{{\rm tr}}

\DeclareMathOperator{\nil}{nil}

\DeclareMathOperator{\Imag}{Im}

\DeclareMathOperator{\diag}{diag}

\newcommand{\SPR}{/\!\!/\!\!/}
\newcommand{\GIT}{/\!\!/}
\DeclareMathOperator{\red}{red}
\DeclareMathOperator{\pf}{pf}

\begin{document}

\title{Towards a symplectic version of the Chevalley restriction theorem}

\author{Michael Bulois}
\address{Michael Bulois\\ Univ Lyon \\ UJM-Saint-Etienne \\ CNRS UMR 5208 \\ Institut Camille Jordan \\ 42023 Saint-Etienne, France}
\email{michael.bulois@univ-st-etienne.fr}

\author{Christian Lehn}
\address{Christian Lehn\\Fakult\"at f\"ur Mathematik\\ Technische Universit\"at Chemnitz\\
Reichenhainer Stra\ss e 39\\09126 Chemnitz, Germany}
\email{christian.lehn@mathematik.tu-chemnitz.de}

\author{Manfred Lehn}
\address{ Manfred Lehn\\
Institut f\"ur Mathematik\\
Johannes Gutenberg--Universit\"at Mainz\\
55099 Mainz, Germany}
\email{lehn@mathematik.uni-mainz.de}

\author{Ronan Terpereau}
\address{Ronan Terpereau\\ Max Planck Institut f\"ur Mathematik\\ Vivatsgasse 7\\ 53111
Bonn, Germany}
\email{rterpere@mpim-bonn.mpg.de}


% 14L30 Group actions on varieties or schemes (quotients) 
% 53D20 Momentum maps; symplectic reduction

% 20G05 Representation theory
% 20G20 Linear algebraic groups over the reals, the complexes, the quaternions
% 13A50 Actions of groups on commutative rings; invariant theory

\begin{abstract}
If $(G,V)$ is a polar representation with Cartan subspace $\gothc$ and
Weyl group $W$, it is shown that there is a natural morphism
of Poisson schemes $\gothc\oplus \gothc^*/W\to V\oplus V^*\SPR G$. This morphism is conjectured to be an isomorphism of the underlying reduced varieties if $(G,V)$ is visible. The conjecture is proved for
visible stable locally free polar representations and some other examples.
\end{abstract} 

\maketitle

\tableofcontents
%%%%%%%%%%%%%%%%%%%%%%%%%%%%%%%%%%%%%%%%%%%%%%%%%%%%%%%%%%%%%%%%%%%%%%%%%%%%%%%
%%%%%%%%%%%%%%%%%%%%%%%%%%%%%%%%%%%%%%%%%%%%%%%%%%%%%%%%%%%%%%%%%%%%%%%%%%%%%%%
\section{Introduction}

The quotient of a symplectic complex vector space $(M,\omega)$ by a 
reductive group $G$ does not naturally inherit a symplectic structure,
even if the group action preserves the symplectic form $\omega$. The classical
approach is to consider the moment map $\mu:M\to \gothg^*$ and the quotient
of its null-fibre $M\SPR G:=\mu^{-1}(0)\GIT G$ instead, the so-called 
\emph{Marsden-Weinstein reduction} or \emph{symplectic reduction}. Such symplectic reductions 
arise naturally as local models for the singularities of certain quiver varieties 
or moduli of sheaves on K3 surfaces. Symplectic reductions are, in general,
rather singular spaces and often fail to satisfy Beauville's criteria for
a symplectic singularity \cite{Beau}. They do, however, always carry natural 
Poisson structures.

In this article, we study a particularly interesting class of symplectic reductions
which in many cases turn out to be isomorphic to quotients of finite groups. Dadok 
and Kac \cite{DK85} introduced the concept of a {\sl polar representation} $(G,V)$ 
of a reductive group. One feature of polar representations is the existence
of a linear subspace $\gothc\subseteq V$ with an action of a finite subfactor $W$
of $G$ such that the inclusion $\gothc\to V$ induces an isomorphism 
$\gothc/W\to V\GIT G$. A classical example for such a situation is the adjoint 
action of $G$ on its Lie algebra $\gothg$: if $W$ denotes its 
Weyl group and $\gothh$ a Cartan subalgebra of $\gothg$, then $\gothh/W\isom 
\gothg\GIT G$ by Chevalley's theorem. Now the representations $\gothc\oplus 
\gothc^*$ and $V\oplus V^*$ carry natural symplectic structures that are 
equivariant for the actions of $W$ and $G$, respectively. As a first technical
result we construct a morphism of Poisson schemes
\begin{equation}\label{eq morphism}
r:\gothc\oplus \gothc^*/W\to V\oplus V^*\SPR G.
\end{equation}
This can be seen as a symplectic analogue of Chevalley's
isomorphism, and it is natural to seek for conditions for $r$ to actually be an isomorphism. 

\begin{conjecture}\label{conj} --- Let $(G,V)$ be a visible polar representation of 
a reductive algebraic group $G$. Then $r:\gothc\oplus \gothc^*/W\to
(V\oplus V^*\SPR G)_{\red}$ is an isomorphism of Poisson varieties, where $(V\oplus V^*\SPR G)_{\red}$ denotes the reduced scheme associated to the symplectic reduction $V\oplus V^*\SPR G$.
\end{conjecture}

Here a representation $(G,V)$ is said to be {\sl visible} if each fibre of the 
quotient map $V\to V\GIT G$ consists of only finitely many orbits. 
The morphism $r:\gothc\oplus\gothc^*/W\to (V\oplus V^*\SPR G)_{\red}$ is known
to be an isomorphism for the adjoint representation 
by results of Richardson \cite{Ric79}, Hunziker \cite{Hun97} (for simple groups), and Joseph \cite{Jos97} (for semisimple groups),
and for isotropy representations associated to symmetric spaces by the work of Panyushev \cite[\S\S~3-4]{Pan94} (for the locally free case) and Tevelev \cite{Tev00} (for the general case). These are special cases of visible polar
representations and belong to the class of $\theta$-representations studied
by Vinberg in \cite{Vin76}. We will review this concept in section \ref{sec:Theta}. 
In the framework of quiver representations, Crawley-Boevey \cite{CB03} has shown that the symplectic reduction is normal, which implies the conjecture for the  $\theta$-representations obtained from quivers. Finally, $r$ is also known to be an isomorphism in some special cases considered in \cite{Leh07}, \cite{Bec09}, and \cite{Ter13}.

In support of the conjecture we cover the following new cases.

\begin{theorem}\label{thm:maintheorem}--- Let $(G,V)$ be a visible stable 
locally free polar representation with Cartan subspace $\gothc$ and Weyl 
group $W$. Then the restriction morphism 
$r: \gothc\oplus\gothc^*/W\lra V\oplus V^*\SPR G$
is an isomorphism of Poisson schemes. 
\end{theorem}

A representation $(G,V)$ is said to be {\sl locally free} if the stabiliser 
subgroup of a general point in $V$ is finite, and {\sl stable} if there is 
a closed orbit whose dimension is maximal among all orbit dimensions. 
The $\theta$-representations are always visible and polar, but are not necessarily
locally free or stable. 
One might wonder how special a visible stable polar representation is, e.g., whether such a representation is automatically a $\theta$-representation. It turns out that this is not the case: in Example \ref{ex visible stable locally free not theta} we provide a family of examples of visible stable locally free polar representations which are not $\theta$-representations.

We first obtained a proof of Theorem \ref{thm:maintheorem} for $\theta$-representations by using their classification, but the proof we present here is more elegant and classification free which is why we are convinced that the framework of polar representations is the one where Conjecture \ref{conj} should be posed.
Further evidence for Conjecture \ref{conj} 
is provided in this paper through a number of examples where we are able to verify the conjecture or, in some cases, even the stronger statement where the symplectic reduction is reduced. For instance, we prove that the conjecture holds for $\theta$-representations with a Cartan subspace of dimension at most $1$ (see Corollary \ref{corollary dim one}); this case is the generic case and could be of major importance in an inductive approach to proving the conjecture for $\theta$-representations.

We will show by counterexamples that the conjecture does not hold in general if the visibility assumption is dropped; see Proposition \ref{prop:nonvisible} and Example \ref{ex nonvisible}. Also, the $16$-dimensional spin 
representation of $\LieSpin_9$, which is the isotropy representation of a symmetric space, provides an example (see \S ~\ref{examples}) where the symplectic reduction is irreducible but non-reduced, whereas, of course, $\gothc\oplus\gothc^*/W$ is always reduced. 
So, unless stronger hypotheses on 
$(G,V)$ are imposed, the symplectic reduction has to be taken with its reduced structure.

Let us note that the question of whether the symplectic reduction $\gothg\oplus\gothg^*\SPR G$ in 
the adjoint case is reduced (resp. normal) from the start is closely related to the classical 
problem of whether the {\sl commuting scheme} 
$$\mu^{-1}(0)=\{(x,y)\in\gothg^2\;|\; [x,y]=0\}$$ is reduced (resp. normal) or not; see for example \cite{Pop08,Pro} for results in this direction.\\

\noindent \textbf{Notation.}
We call an integral separated scheme of finite type over $\mathbb{C}$ a \emph{variety}; in particular, varieties are always irreducible. 
If $X$ is a scheme, then we denote by $X_{\mathrm{red}}$ the corresponding reduced subscheme. 
We always denote by $G$ a reductive algebraic group and by $\gothg$ its Lie algebra .

%%%%%%%%%%%%%%%%%%%%%%%%%%%%%%%%%%%%%%%%%%%%%%%%%%%%%%%%%%%%%%%%%%%%%%%%%%%%%%%
%%%%%%%%%%%%%%%%%%%%%%%%%%%%%%%%%%%%%%%%%%%%%%%%%%%%%%%%%%%%%%%%%%%%%%%%%%%%%%%
\section{Symplectic reductions}

Let $M$ be a finite-dimensional complex vector space with a symplectic bilinear
form $\omega$, and let $G$ be a reductive group that acts linearly on $M$ 
and preserves the symplectic structure. Such an action is always Hamiltonian, i.e.\
it admits a moment map $\mu:M\to \gothg^*$, namely $\mu(m)(A)=\mu^A(m)=
\frac12\omega(m,Am)$ for all $m\in M$ and $A\in\gothg$, where $\gothg$ denotes 
the Lie algebra of $G$.
Since $\mu$ is $G$-equivariant, $G$ acts on the fibre $\mu^{-1}(\zeta)$ for any
$\zeta\in (\gothg^*)^G$. The quotient $\mu^{-1}(\zeta)\GIT G$ is referred to as
the \emph{symplectic reduction} or \emph{Marsden-Weinstein reduction}. In this paper, we are
interested only in the case $\zeta=0$ case, which is the only choice for 
semisimple groups anyway, and we write $M\SPR G:=\mu^{-1}(0)\GIT G$. If the group $G$
is finite, the moment map vanishes identically, and the symplectic reduction 
coincides with the ordinary quotient, $M\SPR G=M/G$.

The symplectic reduction tends to be rather singular. Despite its name, 
it fails in general to satisfy the criteria of a symplectic singularity in the 
sense of Beauville since $M\SPR G$ can be reducible or even non-reduced, 
as several of the examples we discuss in this paper show. However, $M\SPR G$ always 
carries a canonical Poisson structure.

Recall that a \emph{Poisson bracket} on a $\IC$-algebra $A$ is a $\IC$-bilinear Lie 
bracket $\{-,-\}:A\times A\to A$ that satisfies the Leibniz rule 
$\{fg,h\}=f\{g,h\}+g\{f,h\}$. An ideal $I\subseteq A$ is a Poisson ideal if
$\{I,A\}\subseteq I$. Kaledin \cite{Kal06} showed that all minimal
primes of $A$ and the nilradical $\nil(A)$ are Poisson ideals. A \emph{Poisson scheme} 
is a scheme $X$ together with a Poisson bracket on its structure sheaf.

A symplectic form on a smooth variety 
induces a Poisson structure in the following way. The Hamilton vector field 
$H_f$ associated to a regular function $f$ is defined by the relation 
$df(\xi)=\omega(H_f,\xi)$ for all tangent vectors $\xi$. Then $\{f,f'\}:=
df(H_{f'})=-df'(H_f)$. The skew-symmetry of the bracket is immediate, and the 
Jacobi identity follows from $d\omega=0$. If $X$ is a normal symplectic variety, 
the bracket $\{f,f'\}$ is defined as the unique regular extension of the 
function $\{f|_{X_{\reg}},f'|_{X_{\reg}}\}$. For a symplectic vector space
$(M,\omega)$ with a basis $x_1,\ldots,x_{2n}$ and a dual basis $y_1,\ldots, 
y_{2n}$ characterised by $\omega(x_i,y_j)=\delta_{ij}$, the Poisson bracket is 
given by $\{f,f'\}=\sum_{i} \frac{\partial f}{\partial x_i}
\frac{\partial f'}{\partial y_i}$.
 
The ideal $I\subseteq \IC[M]$ of the null-fibre $\mu^{-1}(0)$ of the moment map 
is generated by the functions $\mu^A(m)=\frac12\omega(m,Am)$. Their differentials 
equal $d\mu_m^A(\xi)=\omega(\xi,Am)$. Hence, for any regular function $f$, one 
has $\{f,\mu^A\}(m)=-\omega(H_f,Am)=-df_m(Am)$. If $f$ is assumed to be 
$G$-equivariant, then its derivative vanishes in the orbit directions, so that 
$\{f,\mu^A\}=0$ for all $f\in \IC[M]^G$ and all $A\in \gothg$. The Leibniz rule 
now implies $\{I,\IC[M]^G\}\subseteq I$ and hence $\{I^G, \IC[M]^G\}=I^G$ as the bracket is $G$-invariant. This 
signifies that $I^G$ is a Poisson ideal in the invariant ring $\IC[M]^G$ and, 
consequently, that $W\SPR G=\Spec((\IC[M]/I)^G)$ inherits a canonical Poisson 
structure. Note that $(\IC[M]/I)^G=\IC[M]^G/I^G$ due to the linear reductivity 
of $G$. 

In this paper, the relevant symplectic representations take the special form 
$M=V\oplus V^*$ with symplectic form $\omega(v+\varphi, v'+\varphi')
:=\varphi(v')-\varphi'(v)$. Then for any linear action of $G$ on $V$ and the 
corresponding contragredient action on the dual space $V^*$, the diagonal 
action on $V\oplus V^*$ preserves the symplectic structure. The moment map takes 
the special form $\mu^A(v+\varphi)=\varphi(Av)$ for $v+\varphi\in V\oplus V^*$. 
We refer to $V\oplus V^*$ as the {\sl symplectic double} of $V$.

%%%%%%%%%%%%%%%%%%%%%%%%%%%%%%%%%%%%%%%%%%%%%%%%%%%%%%%%%%%%%%%%%%%%%%%%%%%%%%%
%%%%%%%%%%%%%%%%%%%%%%%%%%%%%%%%%%%%%%%%%%%%%%%%%%%%%%%%%%%%%%%%%%%%%%%%%%%%%%%
\section{Symplectic double of polar representations}

The concept of a {\sl polar representation} was developed and thoroughly 
investigated by Dadok and Kac. As this notion is central for our purposes, we 
briefly recall here the main concepts and results but refer to 
\cite{DK85} for all proofs.

Let $V$ be a representation of a reductive group $G$. A vector $v\in V$ is {\sl 
semisimple} if its orbit is closed. Such orbits are affine, and by Matsushima's 
criterion the stabiliser subgroup $G_v\subseteq G$ of $v$ is again reductive. A 
semisimple $v$ is {\sl regular} if its orbit dimension is maximal among all 
semisimple orbits. For a regular $v$ consider the linear space $\gothc_v:=
\{x\in V\;|\; \gothg x\subseteq \gothg v\}$. Obviously, $\gothc_v=\gothc_w$ for 
all regular semisimple elements $w\in \gothg_v$. 

Dadok and Kac showed that $\gothc_v$ consists of semisimple elements only and
that it is annihilated by the Lie algebra $\gothg_v$ of the stabiliser subgroup 
of $v$ \cite[Lem.~2.1]{DK85}. Moreover, the natural projection $\gothc_v\to V
\GIT G$ is finite \cite[Prop.~2.2]{DK85} so that in particular its dimension 
is bounded by $\dim\gothc_v\leq \dim V\GIT G$. A representation $V$ is said to be 
{\sl polar} if there is a regular element $v$ such that $\dim \gothc_v=
\dim V\GIT G$. The space $\gothc_v$ is then called a {\sl Cartan subspace} of $V$. 
 
For instance, if $\dim V \GIT G \leq 1$, then $V$ is polar and any non-zero semisimple element generates a Cartan subspace. The adjoint representation $\gothg$ of a semisimple Lie group $G$ 
is polar, and any Cartan subalgebra $\gothh\subseteq\gothg$ is a Cartan subspace. 
The natural representation of $\LieSl_3$ on the space $S^3\IC^3=\IC[x,y,z]_3$ 
of ternary cubic forms is polar, and the Hesse family $\langle x^3+y^3+z^3, 
xyz\rangle$ is a Cartan subspace. 

Assume now that $V$ is a polar representation of $G$. By \cite[Thm.~2.3]{DK85},
all Cartan subspaces are $G$-conjugate. Let $\gothc\subseteq V$ be a fixed Cartan 
subspace, and let $v\in\gothc$ denote a regular semisimple element. Dadok and Kac 
further proved that there is a maximal compact Lie group $K\subseteq G$ and a 
$K$-equivariant Hermitian scalar product  $\langle-,-\rangle$ on $V$ such that all 
vectors of $\gothc$ have minimal length in their orbit, and that $\gothc$ and 
$\gothg\gothc=\gothg v$ are orthogonal to each other \cite[Lem.~2.1]{DK85}.
The quotient of the normaliser subgroup $N_G(\gothc)=\{g\in G\;|\; g(\gothc)=
\gothc\}$ and the centraliser subgroup $Z_G(\gothc)=\{g\in G\;|\; g(x)=x
\text{ for all }x\in \gothc\}$ is a finite group $W=N_G(\gothc)/Z_G(\gothc)$ 
which, in analogy to the case of adjoint representations, is called the 
{\sl Weyl group} of the polar representation; $W$ is generated by complex reflections when $G$ is connected. 
Theorem 2.10 of \cite{DK85} 
states that the inclusion $\gothc\to V$ induces an isomorphism
\begin{equation}\gothc/W\xra{\;\isom\;} V\GIT G,\end{equation}
or, equivalently, $\IC[V]^G\isom \IC[\gothc]^W$. This generalises the Chevalley 
isomorphism in the case of adjoint representations. Since $G_v$ stabilises
the tangent space $\gothg v$ to the $G$-orbit of $v$ and since the definition
of $\gothc=\gothc_v$ depends only on $\gothg v$, the stabiliser subgroup 
acts on $\gothc$, so there are natural inclusions $Z_G(\gothc)\subseteq
G_v\subseteq N_G(\gothc)$ of finite index. All three groups are reductive. 

Our goal is to construct a morphism
\begin{equation}r: \gothc\oplus \gothc^*/W\lra V\oplus V^*\SPR G\end{equation}
and to give conditions for $r$ to be an isomorphism. As the 
functorial projection $V^*\to \gothc^*$ points in the wrong direction the first 
task is to identify an appropriate subspace $\gothc\dual\subseteq V^*$ that splits the 
projection.

Let $U\subseteq V$ be the orthogonal complement to $\gothc\oplus\gothg\gothc$.
By \cite[Prop.~1.3 ii)]{DK85}, the orthogonal decomposition
\begin{equation}\label{eq:Udecomposition}
V=\gothc\oplus \gothg\gothc\oplus U
\end{equation}
is $G_v$-stable, and by \cite[Cor.~2.5]{DK85}, $\dim U\GIT G_v=0$. As 
$N_G(\gothc)$ normalises $Z_G(\gothc)$, the decomposition is also preserved
by the normaliser group $N_G(\gothc)$. Let 
\begin{equation}
\gothc\dual:=\{\varphi\in V^*\;|\;\varphi(\gothg\gothc\oplus U)=0\}.
\end{equation} 

\begin{lemma}--- The subspace $\gothc\dual\subseteq V^*$ depends only on $\gothc$. It is stable
under the action of $N_G(\gothc)$ and pointwise invariant under the action of 
$Z_G(\gothc)$. The natural pairing $\gothc\times\gothc\dual\to \IC$ is 
non-degenerate and induces an $N_G(\gothc)$-equivariant isomorphism 
$\gothc\dual\to \gothc^*$.
\end{lemma}

\begin{proof} There is a unique $Z_G(\gothc)$-stable complement $V_1$ to the
invariant subspace $V_0=V^{Z_G(\gothc)}$. By definition, $\gothc\subseteq V_0$, and 
since $\dim U\GIT G_v=0$, one has $U\subseteq V_1$. Hence $\gothg\gothc\oplus U
=\gothg\gothc+V_1$. This shows that $\gothc\dual$ is independent of the choice 
of the Hermitian scalar product. By definition, $\gothc\dual$ pairs trivially 
with $\gothg\gothc\oplus U$, so the pairing with $\gothc$ is non-degenerate. 
The statements about the action of $N_G(\gothc)$ and $Z_G(\gothc)$ follow from 
the $N_G(\gothc)$-equivariance of the decomposition 
\eqref{eq:Udecomposition}. 
\end{proof}

\begin{proposition}--- Let $V$ be a polar representation of a reductive group 
$G$ with Cartan subspace $\gothc$. Then the contragredient representation $V^*$ 
is also polar, and $\gothc\dual$ is a Cartan subspace for $V^*$. The action of 
the Weyl group $W$ on $\gothc\dual$ defines an isomorphism of $W$ with the Weyl 
group of $\gothc\dual$.
\end{proposition}

\begin{proof} The Hermitian scalar product $\langle-,-\rangle$ defines a 
$K$-equivariant semilinear isomorphism 
\begin{equation}\label{eq:Phiintroduced}
\Phi:V\to V^*, v\mapsto \langle v, -\rangle.\end{equation}
Via $\Phi$, $V^*$ inherits a $K$-equivariant Hermitian scalar product
$\langle \varphi,\psi\rangle:=\langle \Phi^{-1}(\psi),\Phi^{-1}(\varphi)\rangle$.
The orthogonality of the decomposition \eqref{eq:Udecomposition} implies that
$\gothc\dual=\Phi(\gothc)$. 
Let $\gothk$ denote the Lie algebra of $K$. The $K$-equivariance of the
Hermitian product implies $\Phi(\gothk w)=\gothk\Phi(w)$ for any $w\in V$. And 
since $\gothg=\IC\gothk$, one also has $\gothg\Phi(w)=\Phi(\gothg w)$. In
particular, for all $w\in \gothc$, $\langle \Phi(w),\gothg \Phi(w)\rangle =0$, so 
$\Phi(w)$ is semisimple and of minimal length in its orbit by the 
Kempf-Ness theorem \cite[Thm.~1.1]{DK85}. We also deduce that $\Phi(v)$ is 
regular and that $\gothc\dual=\{\varphi\in V^*\;|\; \gothg \varphi\subseteq
\gothg \Phi(v)\}=\gothc_{\Phi(v)}$. Any element $g\in G$ may be written in 
the form $g=k\exp(i\xi)$ with $k\in K$ and $\xi\in \gothk$. The $K$-equivariance
and semilinearity of $\Phi$ yield $\Phi(gw)=\Phi(k\exp(i\xi)w)=k\exp(-i\xi)\Phi(w)$
and hence $\Phi(Gw)=G\Phi(w)$ for every $w\in V$. In particular, $\Phi$ 
provides a bijection between closed orbits in $V$ and $V^*$. Necessarily, 
all closed orbits in $V^*$ meet $\gothc\dual$. This implies $\dim \gothc\dual
\geq \dim V^*\GIT G$ and hence that $V^*$ is polar. The statement about the
action of $W$ follows from the fact that the Weyl group can be seen as the
quotient $N_K(\gothc)/Z_K(\gothc)$ according to \cite[Lem.~2.7]{DK85}.
\end{proof}

We continue to use the $K$-equivariant semilinear automorphism $\Phi:V\to V^*$
of \eqref{eq:Phiintroduced}.
The proof of the proposition shows that $\Phi(\gothc)=\gothc\dual$ and 
$\Phi(\gothg\gothc)=\gothg\gothc\dual$. One obtains a $K$-equivariant orthogonal 
decomposition
\begin{equation}\label{eq:dualdecomposition}
V^*=\gothc\dual\oplus \gothg\gothc\dual \oplus U\dual
\end{equation}
with $U\dual:=\Phi(U)$ dual to \eqref{eq:Udecomposition}. As before,
we consider the symplectic form $\omega(v+\varphi, v'+\varphi'):=\varphi(v')
-\varphi'(v)$ on $V\oplus V^*$. The symplectic vector space $(V\oplus V^*,\omega)$
splits into the direct sum of symplectic subspaces $\gothc\oplus \gothc\dual$,
$\gothg\gothc\oplus \gothg\gothc\dual$ and $U\oplus U\dual$. As 
$\omega(\gothc\oplus \gothc\dual, \gothg\gothc\oplus\gothg\gothc\dual)=0$, 
it follows that 
\begin{equation}
\gothc\oplus\gothc\dual\subseteq \mu^{-1}(0),
\end{equation}
where $\mu:V\oplus V^*\to \gothg^*$ denotes the moment map. 

Polar representations behave well under taking slices \cite[\S~2]{DK85}. 
Let $w\in \gothc$ be any element. The orthogonal complement $N_w$ to the
tangent space $\gothg w\subseteq \gothg\gothc$ is stable under $G_w$ and 
contains $\gothc$. The representation of $G_w$ on $N_w$ is called the {\sl slice
representation} at $w$, and Dadok and Kac have shown that $N_w$ is again polar
and that $\gothc$ is a Cartan subspace.

\begin{proposition}\label{prop:finiteinjective}
--- Let $V$ be a polar representation of $G$ with a Cartan 
subspace $\gothc\subseteq V$ and corresponding Weyl group $W$. Let $\gothc\dual
\subseteq V^*$ denote the dual Cartan subspace. Then every $m\in \gothc\oplus 
\gothc\dual$ has a closed $G$-orbit in $V\oplus V^*$ and
$Gm\cap (\gothc\oplus \gothc\dual)=Wm$. Also, the morphism
$$r: \gothc\oplus \gothc\dual/W\lra V\oplus V^*\SPR G$$
is injective and finite.
\end{proposition}

\begin{proof} The symplectic double $V\oplus V^*$ comes with a $K$-equivariant
Hermitian scalar product 
$$\langle w+\varphi, w'+\varphi'\rangle=
\langle w,w'\rangle+\langle \varphi,\varphi'\rangle$$ 
inherited from the scalar products on its direct summands. The orthogonality
of the decompositions \eqref{eq:Udecomposition} and \eqref{eq:dualdecomposition}
implies that $\langle m,\gothg m\rangle=0$ for any $m\in\gothc\oplus \gothc\dual$.
It follows from the Kempf-Ness criterion \cite[Thm.~1.1]{DK85} that $m$ is 
semisimple and of minimal length in its $G$-orbit, and that 
$$Gm\cap( \gothc\oplus
\gothc\dual) =Km\cap (\gothc\oplus\gothc\dual).$$
Assume now that $(w,\varphi),(w',\varphi')\in\gothc\oplus \gothc\dual$ belong
to the same $G$-orbit and hence the same $K$-orbit, say $w'=kw$ and $\varphi'=k\varphi$
for some $k\in K$. Up to the action of $W$, we may assume $w'=w$, so that 
$k\in K_w=G_w\cap K$. Let $x=\Phi^{-1}(\varphi),x'=\Phi^{-1}(\varphi')\in\gothc$.
As $\Phi$ is $K$-equivariant, it follows that $kx=x'$. Now $x$ and $x'$ are 
elements in the same $G_w$-orbit and contained in a Cartan subspace $\gothc$
of the slice representation $N_w$. By \cite[Thm.~2.8]{DK85}, $x$ and $x'$ also belong
to the same orbit under the normaliser subgroup $N_{G_w}(\gothc)\subseteq N_G(\gothc)$.
Hence there is an element $\gamma\in W$ with $w'=\gamma w=w$ and $x'=\gamma x$.
In particular, the natural morphism $r: \gothc\oplus \gothc\dual/W\to V\oplus V^*\GIT G$
is injective. The same arguments as in the proof of \cite[Prop.~2.2]{DK85}
show that $r$ is finite.
\end{proof}

If the Cartan subspace is one-dimensional, a simple argument gives something much
stronger.

\begin{proposition}\label{prop:rank1case} --- Let $(G,V)$ be a polar 
representation with a one-dimensional Cartan subspace. Then 
$r:\gothc\oplus\gothc\dual/W\to V\oplus V^*\SPR G$ is a closed immersion.
\end{proposition}

\begin{proof} If $\dim \gothc=1$, the Weyl group $W$ is cyclic, say of order 
$m$, and a generator acts on $\gothc\oplus \gothc\dual=\IC^2$ via 
$(x,y)\mapsto (\zeta x, \zeta^{-1} y)$ for some primitive $m$-th root of unity 
$\zeta$. Hence $\IC[\gothc\oplus\gothc\dual]^W\isom \IC[x^m,xy,y^m]$. We have to show that the restriction morphism $(\IC[V]\otimes \IC[V^*])^G\isom(\IC[V\oplus V^*])^G\to \IC[\gothc\oplus\gothc\dual]^W$ is surjective. Since $V$ 
is a polar representation, there are isomorphisms $\IC[V]^G\to \IC[\gothc]^W$ and 
$\IC[V^*]^G\to \IC[\gothc\dual]^W$. Also, the $G$-invariant pairing 
$V\tensor V^*\to\IC$ restricts to the invariant $xy$. This shows that 
$\IC[V\oplus V^*]^G\to \IC[\gothc\oplus\gothc\dual]^W$ is surjective.
\end{proof}

The null-fibre $\mu^{-1}(0)$ of the momentum map contains the $G$-variety
$$C_0:=\overline{G.(\gothc\oplus \gothc\dual)},
$$
which will play a key role in our analysis.

\begin{proposition}\label{prop:ontoC0}
--- The quotient $C_0\GIT G$ is a Poisson subscheme in $(V\oplus V^*)\GIT G$
and hence also in the symplectic reduction $(V\oplus V^*)\SPR G$. Moreover,
$r: \gothc\oplus \gothc\dual/W\to C_0\GIT G$ is a bijective morphism of 
Poisson schemes. 
\end{proposition}

\begin{proof} Let $J\subseteq \IC[V\oplus V^*]$ denote the vanishing ideal of 
$C_0$. If we can show that $\{h,f\}|_{C_0}=0$ for all $h\in J$ and all 
$f\in \IC[V\oplus V^*]^G$, then the Leibniz rule implies $\{J,\IC[V\oplus V^*]^G\}
\subseteq J$, and hence $\{J^G,\IC[V\oplus V^*]^G\}\subseteq J^G$, which covers the 
first assertion. Moreover, as $J$ is a $G$-equivariant ideal sheaf and 
$G(\gothc\oplus \gothc\dual)$ is dense in $C_0$ by definition, it suffices to
 show that $\{h,f\}$ vanishes in a general point $m=w+\varphi$ of $\gothc\oplus 
\gothc\dual$ for all $h\in J$ and $f\in\IC[V\oplus V^*]^G$. Now the tangent space 
of $V\oplus V^*$ decomposes into symplectic subspaces $\gothc\oplus \gothc\dual$, 
$\gothg\gothc\oplus \gothg\gothc\dual$ and $U\oplus U\dual$, and this decomposition 
is stable under $G_w$. According to a result of Dadok and Kac mentioned above,
the quotient $U\GIT G_w$ is zero-dimensional. This implies that any $G$-invariant 
function $f$ is constant on subsets $m+U$ and $m+U^\perp$. In particular, all 
derivatives of $f$ in $m$ in the directions $U\oplus U\dual$ vanish. On the other 
hand, $h$ vanishes on $C_0$ so that all derivatives of $h$ vanish in the directions 
$\gothc\oplus \gothc\dual$. Thus, the calculation of the Poisson bracket is 
reduced to 
\begin{equation}
\{h,f\}=\sum_{i=1}^{2\ell} \frac{\partial h}{\partial x_i}\frac{\partial f}{\partial y_i},
\end{equation}
where the $x_i$ and the $y_i$ run through a basis and the dual
basis of $\gothg\gothc\oplus \gothg\gothc\dual$, respectively. For a general $m$, the tangent
space $\gothg m\subseteq \gothg\gothc\oplus \gothg\gothc\dual$ is a Lagrangian
subspace: as it is half-dimensional, it suffices to verify that it is isotropic.
Indeed, for any $A,B\in\gothg$, one has 
\begin{equation}
\{Am,Bm\}=\{Aw+A\varphi,Bw+B\varphi\}=A\varphi(Bw)-B\varphi(Aw)=\varphi([B,A]w)=0,
\end{equation}
since $\gothc\dual$ annihilates $\gothg\gothc$ by construction. Let $x_1,\ldots,x_\ell$ 
be a basis of $\gothg m$ and augment it by $x_{\ell+1},\ldots,x_{2\ell}$ to 
form a symplectic basis of $\gothg\gothc\oplus \gothg\gothc\dual$. Then
\begin{equation}
\{h,f\}=\sum_{i=1}^\ell \left(\frac{\partial h}{\partial x_i}
\frac{\partial f}{\partial x_{\ell+i}}- \frac{\partial f}{\partial x_i}
\frac{\partial h}{\partial x_{\ell+i}}\right).\end{equation}
Since both functions $f$ and $h$ are constant along $Gm$, the partial derivatives
$\partial f/\partial x_i$ and $\partial h/\partial x_i$ vanish in $m$, so that 
finally $\{h,f\}(m)=0$.
 
For the last statement, it suffices to show that for any two functions 
$f, f'\in \IC[V\oplus V^*]^G$ one has 
$$\{f,f'\}|_{\gothc\oplus \gothc\scrdual}=
\{f|_{\gothc\oplus \gothc\scrdual}, f'|_{\gothc\oplus\gothc\scrdual}\},$$
where the bracket on the left is that in $\IC[V\oplus V^*]$ and the bracket on 
the right is that in $\IC[\gothc\oplus \gothc\dual]$. Splitting the sum
$$\{f,f'\}=\sum_i \frac{\partial f}{\partial x_i}\frac{\partial f'}{\partial y_i}
$$
into contributions from $\gothc\oplus \gothc\dual$, $\gothg\gothc\oplus\gothg\gothc\dual$
and $U\oplus U\dual$, this amounts to proving that the latter two summands 
only contribute trivially. But this follows from the same arguments as before:
the derivatives of both $f$ and $f'$ vanish in the directions $U\oplus U\dual\oplus 
\gothg m$ due to their $G$-equivariance.

The bijectivity of $r$ is obvious from Proposition \ref{prop:finiteinjective} and the definition of $C_0$.
\end{proof}

From the two propositions above it follows that the conjecture holds if and only if the 
following assertions are true.
\begin{enumerate}
\item The inclusion $C_0\subseteq\mu^{-1}(0)$ induces an isomorphism $C_0\GIT G
\to (V\oplus V^*\SPR G)_{\red}$. This is equivalent to saying that $r:\gothc
\oplus \gothc\dual/W\to V\oplus V^*\SPR G$ is bijective.
\item The variety $C_0\GIT G$ is normal.
\end{enumerate}

\begin{remark} ---
If $(G,V)$ is a stable polar representation such that $\mu^{-1}(0)$ is a normal variety, then the conjecture holds (see Lemma~\ref{lem:IrreducibilityCriterion}).
However, in general $\mu^{-1}(0)$ is not even irreducible; see for instance \cite[\S 4]{PY07} for examples where $\mu^{-1}(0)$ has an arbitrarily large number of irreducible components. 
\end{remark}

%%%%%%%%%%%%%%%%%%%%%%%%%%%%%%%%%%%%%%%%%%%%%%%%%%%%%%%%%%%%%%%%%%%%%%%%%%%%%%%
%%%%%%%%%%%%%%%%%%%%%%%%%%%%%%%%%%%%%%%%%%%%%%%%%%%%%%%%%%%%%%%%%%%%%%%%%%%%%%%
\section{The structure of the null-fibre of the moment map}

As before, let $V$ denote a representation of a reductive group $G$, let 
$\pi:V\to V\GIT G$ be the quotient map, and let $\mu:V\oplus V^*\to \gothg^*$ be the 
moment map. The aim of this section is to highlight some properties of the null-fibre $\mu^{-1}(0)$. Part of this content is directly inspired by \cite{Pan94}.

Consider the linear map $f: V\times\gothg\to V\times V$, $(x,A)\mapsto (x,Ax)$
of trivial vector bundles on $V$. For each point $x\in A$, one has $\rk(f(x))=
\dim \gothg x=\dim Gx$. A {\sl sheet} $S$ of $V$ is an irreducible component
of any of the strata $\{x\in V\;|\; \rk f(x)=r\}$, $r\in \IN_0$, and one puts
$r_S:=\rk f(x)$ for any $x\in S$. The number $\mod(G,S)=\dim S-r_S$ is called 
the {\sl modality} of $S$. The restriction $f|_S$ of $f$ to a 
sheet $S$ with its reduced subscheme structure has constant rank. Hence its 
image and its kernel are subbundles of rank $r_S$ and rank $\dim\gothg-r_S$, 
respectively. 
Let $\pr_1:\mu^{-1}(0)\to V$ denote the projection to the first component of 
$V\oplus V^*$. Then 
$$
\pr_1^{-1}(S)=\{(x,\varphi)\in S\times V\;|\; \varphi\perp \gothg x=\Imag f(x)\}
$$
is a subbundle in $S\times V^*$ of rank $\dim V^* -r_S$. In particular, it is 
an irreducible locally closed subset of $\mu^{-1}(0)$ of dimension 
\begin{equation}\label{eq:contributionfromsheetS}
\dim \pr_1^{-1}(S)=\dim S+\dim V^*-r_S=\dim V +\mod(G,S).
\end{equation}
Clearly, $\mu^{-1}(0)=\bigcup_S \pr_1^{-1}(S)$, and
since the set of sheets $S$ of $V$ is finite, it follows that
$$\dim \mu^{-1}(0)=\dim V +\max_S \mod(G,S).$$

A representation $(G,V)$ is  said to be {\sl visible} if each fibre of $\pi$ 
has only finitely many orbits. It is well known that it suffices to require
that the special fibre $\pi^{-1}(\pi(0))$ has only finitely many orbits.

\begin{lemma}\label{lem:VisibilityandModality} --- 
If $V$ is visible, then $\mod(G,S)=\dim \pi(S)$ for each sheet $S$ of $V$.
\end{lemma}

\begin{proof} Visibility of $V$ implies that the fibres of the restriction $\pi|_S$
have only finitely many orbits, all of dimension $r_S$. Hence all fibres
of $\pi|_S$ have dimension $r_S$, so that $\dim S=r_S+\dim \pi(S)$. 
\end{proof}

Let $(G,V)$ be a polar representation with Cartan subspace $\gothc$. The 
space $V$ always contains two special sheets, as described below. 

Let $r'$ be the maximal dimension of an orbit in $V$. By semicontinuity, the set 
$S'$ of points with $r'$-dimensional orbits is an open and hence irreducible 
subset of $V$, the open sheet of $V$. If follows that $C':=\overline{\pr_1^{-1}
(S')}$ is an irreducible component of the null-fibre $\mu^{-1}(0)$
of dimension $2\dim V-r'$.

On the other hand let $r_0$ denote the maximal dimension of a closed orbit. 
Let $S_0$ be a sheet that contains the open subset $\gothg_{\reg}\subseteq\gothc$
of regular semisimple elements, i.e.\ those with $r_0$-dimensional orbits.
Since $\pi:\gothc_{\reg}\to V\GIT G$ is dominant, so is $\pi: S_0\to V\GIT G$. 
Hence, for a general point $s\in S_0$, there exists a point $x\in \gothc_\reg$
with $\pi(s)=\pi(x)$, where $\gothc_\reg$ denotes the dense open subset of $\gothc$ formed by regular elements. The orbit of $x$ is closed and hence contained in the 
closure of the orbit of $s$. But the two orbits have the same dimension and thus 
must be equal. This shows that $G\gothc_{\reg}\subseteq S_0\subseteq \overline{G\gothc}$. 
In particular, the sheet $S_0$ is uniquely determined.

In general, we have $r_0\leq r'$. A representation $(G,V)$ is said to be 
{\sl stable} if the maximal orbit dimension is attained by orbits of regular 
semisimple elements, i.e.\ if $r_0=r'$. 

\begin{lemma}--- Let $(G,V)$ be a stable polar representation. Then
$C_0$ is an irreducible component of $\mu^{-1}(0)$ of dimension $\dim V
+\dim \gothc$.
\end{lemma}

\begin{proof} If $(G,V)$ is stable, the regular sheet $S_0$ and the open sheet 
$S'$ coincide. Also, in this case, the general fibre of the quotient $\pi:V\to 
V\GIT G$ is closed. Therefore, the modality of the sheet $S_0$ equals 
$\dim \gothc$. Hence $\overline{\pr_1^{-1}(S_0)}$ is an irreducible component 
of $\mu^{-1}(0)$ of dimension $\dim V+\dim \gothc$. It remains to show that 
the inclusion $C_0\subseteq \pr_1^{-1}(S_0)$ is an equality. Indeed, if $V$ is 
stable, the space $U$ in the decomposition \eqref{eq:Udecomposition} is trivial 
according to \cite[Cor.~2.5]{DK85}. Thus, $G(\gothc_{\reg}\times\gothc\dual)$ 
is a dense subset of $\pr^{-1}(S_0)$, and $C_0=\overline{\pr^{-1}(S_0)}$.
\end{proof}

\begin{proposition}--- Let $(G,V)$ be a visible polar representation. Then the following properties hold.
\begin{enumerate}
\item $\dim \mu^{-1}(0) =\dim V +\dim \gothc $. 
\item The irreducible components of $\mu^{-1}(0)$ of maximal dimension are 
in bijection with sheets $S$ such that $\pi(S)\subseteq V\GIT G$ is dense.
\item If $(G,V)$ is stable, then $C_0$ is the only irreducible 
component of $\mu^{-1}(0)$ of maximal dimension.
\item If $(G,V)$ is unstable, then $\mu^{-1}(0)$ has several 
irreducible components of maximal dimension.
\end{enumerate}
\end{proposition}

\begin{proof} The first assertion is a consequence of Panyushev's results 
\cite[Thm.~2.3, Cor.~2.5, Thm.~3.1]{Pan94}. Alternatively, it follows from Lemma 
\ref{lem:VisibilityandModality} that the modality of each sheet is bounded by 
$\dim V\GIT G=\dim \gothc$. This bound is attained by $S_0$. The same argument
proves the second assertion.

If in addition $V$ is stable, then $S_0$ is the only sheet whose image
under $\pi$ is dense in $V\GIT G$. If, on the other hand, $V$ is unstable,
then $S'$ and $S_0$ are distinct and dominate $V\GIT G$, contributing 
two different irreducible components of maximal dimension $d$.
\end{proof}

\begin{example}--- We list a few examples that illustrate the decomposition
of the null-fibre into irreducible components.\label{example_non_red}

1. The standard representation of $\LieSp(V)$ on $V$ is polar
with zero-dimensional Cartan subspace. There are only two sheets: The regular 
sheet $S_0=\{0\}$ and the open sheet $S_1=V\setminus\{0\}$, each forming a
single orbit, so that $V$ is certainly visible. Both sheets contribute
an irreducible component of maximal dimension to $\mu^{-1}(0)=V\times\{0\}\cup 
\{0\}\times V^*$, and $C_0=\{(0,0)\}$ is their intersection.

2. Let $V=V_0\oplus V_1\oplus V_2$ be a $\IZ/3\IZ$-graded vector space with 
components of dimensions $\dim V_0=1$ and $\dim V_1=\dim V_2=2$, and consider 
the corresponding decomposition of the Lie algebra $\Liegl(V)=\gothg_0
\oplus \gothg_1\oplus \gothg_2$. Then $\gothg_0$ is the Lie algebra of $G_0=\LieGl(V_0)\times\LieGl(V_1)\times \LieGl(V_2)$, and
$\gothg_1=\Hom(V_0,V_1)\oplus \Hom(V_1,V_2)\oplus \Hom(V_2,V_0)$ is a visible
unstable polar representation of $G_0$. This is an example of a $\theta$-representation
to be discussed in \S \ref{sec:Theta}. The invariant ring $\IC[\gothg_1]^{G_0}$ is generated
by the function $(a,b,c)\mapsto \tr(abc)$, and if $a:V_0\to V_1$, $b:V_1\to V_2$ and $c:V_2\to V_0$ are rank-one maps such that $\tr(abc)$
is non-zero then the space spanned by $(a,b,c)$ is a Cartan subspace. An explicit
calculation shows that $\mu^{-1}(0)$ has eight irreducible components, six of
dimension 8 and two of dimension 9. The latter two intersect in $C_0$, which is
of dimension $8$. The symplectic reduction $V\oplus V^*\SPR G$ is an $A_2$-singularity. Since
$\gothg_1$ can also be seen as a space of representations of a quiver, the 
results of Crawley-Boevey apply and the normality of $V\oplus V^*\SPR G$ 
is given by \cite[Thm.~1.1]{CB03}.

3. The representation of $\LieSl_3$ on $V=\IC^{3\times 3}$ is stable and polar 
but non-visible. The null-fibre of the moment map equals the space of
pairs of $3\times 3$ matrices $(a,b)$ such that $ab=\frac13\tr(ab)I_3$. It is
11-dimensional with two components of dimension 11 and one component $C_0$ of
dimension 10. The symplectic reduction $V\oplus V^*\SPR \LieSl_3$ is non-reduced.
It has two irreducible components of dimensions 4 and 2, the latter
being $C_0\GIT \LieSl_3$, an $A_2$-surface singularity. In particular, the normality result of \cite{CB03} no longer holds in the general setting of polar representations.
\end{example}

\begin{lemma}\label{lem:IrreducibilityCriterion}
--- Let $(G,V)$ be a stable polar representation. If
$V\oplus V^*\SPR G$ is irreducible then $C_0\GIT G=(V\oplus V^*\SPR G)_{\red}$.
\end{lemma}
 
\begin{proof} Assume that $V\oplus V^*\SPR G$ is irreducible. Then 
there is an irreducible component $C_1\subseteq\mu^{-1}(0)$ dominating
the symplectic reduction. Since every fibre of the quotient map $\pi$
contains exactly one closed orbit, all closed orbits must be contained
in $C_1$. As closed orbits are dense in $C_0$, one has $C_0\subseteq C_1$.
Since by assumption $(G,V)$ is stable polar, $C_0$ is itself an irreducible 
component and hence equals $C_1$. 
\end{proof}

\begin{example}--- Without the stability assumption the conclusion of the 
lemma can be wrong as the following example shows: $(\LieSl_3,\IC^3\oplus \IC^3)$
is a non-visible and unstable polar representation with trivial Cartan subspace.
Consequently, $C_0=\{0\}$. However, $\mu^{-1}(0)\GIT \LieSl_3$ is an irreducible
non-reduced surface, and its reduction is an $A_1$-singularity. 
\end{example}

\begin{lemma}\label{lem:MovingPairs}
--- Let $(G,V)$ be a polar representation with Cartan subspace
$\gothc\subseteq V$ and dual Cartan subspace $\gothc\dual\subseteq V^*$. If 
$(x,y)\in \mu^{-1}(0)$ with semisimple elements $x\in V$ and $y\in V^*$, then
there is an element $g\in G$ such that $(gx,gy)\in\gothc\oplus\gothc\dual$.
\end{lemma}

\begin{proof} As $x$ is semisimple there is an element $h\in G$ such that
$hx\in\gothc$. Replacing $(x,y)$ by $(hx,hy)$, we may assume that $x\in \gothc$.
According to the definition of the moment map, $\mu(x,y)=0$ means that $y$
is contained in the annihilator of the tangent space $\gothg x$, which is
exactly the dual $N_x^*$ of the slice representation $N_x$. As $y$ is semisimple
and since the slice representation has $\gothc$ as a Cartan subspace, there
is an element $g\in G_x$ such that $gy\in\gothc\dual$.
\end{proof}

%%%%%%%%%%%%%%%%%%%%%%%%%%%%%%%%%%%%%%%%%%%%%%%%%%%%%%%%%%%%%%%%%%%%%%%%%%%%%%%
%%%%%%%%%%%%%%%%%%%%%%%%%%%%%%%%%%%%%%%%%%%%%%%%%%%%%%%%%%%%%%%%%%%%%%%%%%%%%%%
\section{Recapitulation of \texorpdfstring{$\theta$}{theta}-representations}\label{sec:Theta}

A particular class of polar representations is formed by the so-called 
$\theta$-representations as introduced by Vinberg \cite{Vin76}. Let $G$ be a connected
reductive group with Lie algebra $\gothg$. Let $\theta$ denote both an 
automorphism of $G$ of finite order $m$ and the induced automorphism
of $\gothg$. After fixing a primitive $m$-th root of unity $\xi$, the action of
$\theta$ gives rise to a $\IZ/m\IZ$-grading $\gothg=\bigoplus_{i\in \IZ/m\IZ}
\gothg_i$ with components $\gothg_i=\Kern(\theta-\xi^i\id)$. Then $\gothg_0$
is a reductive Lie algebra, each $\gothg_i$ is a $\gothg_0$ representation, 
and $\gothg_i$ is dual to $\gothg_{-i}$ via the Killing form. 
Let $G_0\subseteq G$ be the connected
algebraic subgroup with Lie algebra $\gothg_0$. Then $G_0$ is reductive 
and acts linearly on $\gothg_1$. The representation $(G_0,\gothg_1)$ is 
called the $\theta$-representation obtained from the pair $(\gothg,\theta)$.
These representations are visible and polar but not always stable. A complete
list of {\sl irreducible} $\theta$-representations with their main features 
can be found in the paper \cite{Kac80} of Kac. 

A Cartan subspace $\gothc\subseteq\gothg_1$ is defined as a maximal abelian subspace 
consisting of semisimple elements. The subspace $\gothc$ lies in some $\theta$-invariant Cartan subalgebra $\gothh$ of $\gothg$,
and the dual Cartan subspace is $\gothc\dual=\gothh_{-1}\subseteq \gothg_{-1}$.
The dual Cartan subspace can also be described as the space of all semisimple
elements in $\gothg_{-1}$ that commute with $\gothc$. Note that the moment
map $\mu:\gothg_1\oplus\gothg_{-1}\to \gothg_0$ is, up to the identification
via the Killing form, the ordinary commutator. Hence $\mu(x,y)=0$ if and 
only if $x$ and $y$ commute.

If $\gothg$ is a reductive algebra and $x\in\gothg$ is any element, the centraliser
of $x$ in $\gothg$ will be denoted by $\gothg^x=\{y\in\gothg\;|\; [x,y]=0\}$.

In \cite[\S 3.8, Thm.]{Pan94}, Panyushev outlines the proof of the next proposition.
However, he assumes the stability of the $\theta$-representation without actually
using it for the part of the proof relevant to our proposition; also, he considers only
the component $C_0$ instead of $\mu^{-1}(0)$. Therefore, we have opted to 
reformulate his theorem in our setting and to write out the details.

\begin{proposition} \label{prop:theta}
--- Let $(G_0,\gothg_1)$ be a $\theta$-representation with
Cartan subspace $\gothc$ and let $C_0=\overline{G_0(\gothc\oplus \gothc\dual)}$.
Then $C_0\GIT G_0=\mu^{-1}(0)_{\red}\GIT G_0$.
\end{proposition}

\begin{proof} For any $x\in \gothg$, denote by $x_s$ and $x_n$ its semisimple
and nilpotent parts respectively. One can check that $x_s,x_n\in\gothg_i$
if $x\in \gothg_i$. 

Assume now that $\mu(x,y)=0$, i.e.\ that $[x,y]=0$. Then all components $x_s$,
$x_n$, $y_s$ and $y_n$ mutually commute. According to Lemma \ref{lem:MovingPairs}, 
$(x_s,y_s)\in G_0(\gothc\times\gothc\dual)$. Therefore, in order to prove the
proposition, it suffices to show that $(x_s,y_s)\in \overline{G_0(x,y)}$. 
The Lie subalgebra $\gothl=\gothg^{x_s}\cap \gothg^{y_s}$ is reductive by
Matsuhima's criterion, and since $x_s$ and $y_s$ are homogeneous, $\gothl$
is a graded subalgebra, $\gothl=\bigoplus_{i\in \IZ/m\IZ} \gothl_i$. Let 
$L_0\subseteq G_0$ be the connected subgroup of $G_0$ with Lie algebra $\gothl_0$.
We claim that $(0,0)\in \overline{L_0(x_n,y_n)}$. The assertion will then
follow since $\overline{L_0(x,y)}=(x_s,y_s)+\overline{L_0(x_n,y_n)}$. 

By the graded version of the Jacobson-Morozov theorem \cite[\S 2]{Kac80}, there
exist elements $h\in \gothl_0$ and $y'\in\gothl_{-1}$ such that $(x_n,h,y')$
is an $\Liesl_2$-triplet. Consider the corresponding weight space decomposition
$$\gothl=\bigoplus_{j\in \IZ} \gothl_h(j),\quad \gothl_h(j)=\{ z\in\gothl\;|\;
[h,z]=jz\}$$
and the associated one-parameter subgroup $\phi:\IG_m\to L_0$ with
$\phi(t)|_{\gothl_h(j)}=t^j\id_{\gothl_h(j)}$. Since $x_n\in\gothl_h(2)$, we
have $\lim_{t\to 0}\phi(t).x_n=0$. On the other hand, the centraliser of
$x_n$ in $\gothl$, including $y_n$, must be contained in $\bigoplus_{j\geq 0}
\gothl_h(j)$, and so $\lim_{t\to 0}\phi(t).y_n=:y_0$ equals the weight-$0$
component of $y_n$. This shows that $(0,y_0)\in \overline{L_0(x_n,y_n)}$.
Finally, since $y_n$ is nilpotent, the same is true for $y_0$, so that
$(0,0)\in \overline{L_0(0,y_0)}$. 
\end{proof}

\begin{corollary}\label{corollary dim one}--- 
Conjecture \ref{conj} holds whenever $(G_0,\gothg_1)$ has a Cartan subspace of dimension at most $1$.
\end{corollary}
\begin{proof}
If $\dim \gothc=0$, then the symplectic reduction associated to $(G_0,\gothg_1)$ is a point. If $\dim \gothc=1$, then the result follows from Proposition \ref{prop:rank1case} and Proposition \ref{prop:theta}.
\end{proof}

\begin{remark}--- For visible polar representations, such all
$\theta$-representations, having a zero-dimen\-sion\-al Cartan subspace is equivalent to having a finite number of orbits. A list of all irreducible visible
representations with a finite number of orbits can be found in the paper 
\cite[\S 3]{Kac80} of Kac. 
\end{remark}

\begin{remark}\label{rk_normal_muz}--- 
In contrast to the case of Lie algebras and symmetric Lie algebras cases, there exist many stable and locally free $\theta$-representations with a one-dimensional Cartan subspace. It turns out that $\mu^{-1}(0)$ is non-normal only in five of these cases. It is remarkable that, except in these five sporadic cases, the null-fiber of the moment map of a stable and locally free $\theta$-representation obtained from a simple Lie algebra is always normal. Let us describe these $\theta$-representations in terms of their Kac diagrams. Here the number $m$ is the order of the automorphism $\theta$.
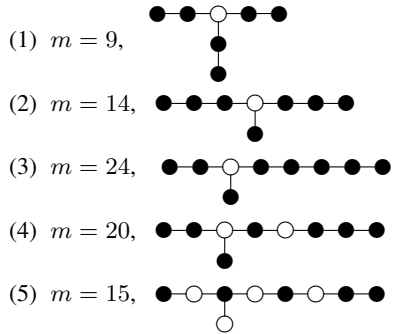
\begin{enumerate}
\item $m=9$, \quad
\unitlength0.5mm\begin{picture}(60, 15)
\put(0,7){
\put(1,2){\circle*{4}}\put(3,2){\line(1,0){4}}
\put(9,2){\circle*{4}}\put(11,2){\line(1,0){4}}
\put(17,2){\circle{4}}\put(19,2){\line(1,0){4}}
\put(25,2){\circle*{4}}\put(27,2){\line(1,0){4}}
\put(33,2){\circle*{4}}\put(17,0){\line(0,-1){4}}
\put(17,-6){\circle*{4}}\put(17,-8){\line(0,-1){4}}
\put(17,-14){\circle*{4}}}
\end{picture}
\item $m=14$,\quad \begin{picture}(60, 15)
\put(1,2){\circle*{4}}\put(3,2){\line(1,0){4}}
\put(9,2){\circle*{4}}\put(11,2){\line(1,0){4}}
\put(17,2){\circle*{4}}\put(19,2){\line(1,0){4}}
\put(25,2){\circle{4}}\put(27,2){\line(1,0){4}}
\put(33,2){\circle*{4}}\put(35,2){\line(1,0){4}}
\put(41,2){\circle*{4}}\put(43,2){\line(1,0){4}}
\put(49,2){\circle*{4}}\put(25,0){\line(0,-1){4}}
\put(25,-6){\circle*{4}}
\end{picture}
\item $m=24$, \quad\begin{picture}(60, 15)
\put(1,2){\circle*{4}}\put(3,2){\line(1,0){4}}
\put(9,2){\circle*{4}}\put(11,2){\line(1,0){4}}
\put(17,2){\circle{4}}\put(19,2){\line(1,0){4}}
\put(25,2){\circle*{4}}\put(27,2){\line(1,0){4}}
\put(33,2){\circle*{4}}\put(35,2){\line(1,0){4}}
\put(41,2){\circle*{4}}\put(43,2){\line(1,0){4}}
\put(49,2){\circle*{4}}\put(51,2){\line(1,0){4}}
\put(57,2){\circle*{4}}\put(17,0){\line(0,-1){4}}
\put(17,-6){\circle*{4}}
\end{picture}
\item $m=20$,\quad \begin{picture}(60, 15)
\put(1,2){\circle*{4}}\put(3,2){\line(1,0){4}}
\put(9,2){\circle*{4}}\put(11,2){\line(1,0){4}}
\put(17,2){\circle{4}}\put(19,2){\line(1,0){4}}
\put(25,2){\circle*{4}}\put(27,2){\line(1,0){4}}
\put(33,2){\circle{4}}\put(35,2){\line(1,0){4}}
\put(41,2){\circle*{4}}\put(43,2){\line(1,0){4}}
\put(49,2){\circle*{4}}\put(51,2){\line(1,0){4}}
\put(57,2){\circle*{4}}\put(17,0){\line(0,-1){4}}
\put(17,-6){\circle*{4}}
\end{picture}
\item $m=15$,\quad \begin{picture}(60, 15)
\put(1,2){\circle*{4}}\put(3,2){\line(1,0){4}}
\put(9,2){\circle{4}}\put(11,2){\line(1,0){4}}
\put(17,2){\circle*{4}}\put(19,2){\line(1,0){4}}
\put(25,2){\circle{4}}\put(27,2){\line(1,0){4}}
\put(33,2){\circle*{4}}\put(35,2){\line(1,0){4}}
\put(41,2){\circle{4}}\put(43,2){\line(1,0){4}}
\put(49,2){\circle*{4}}\put(51,2){\line(1,0){4}}
\put(57,2){\circle*{4}}\put(17,0){\line(0,-1){4}}
\put(17,-6){\circle{4}}
\end{picture}\\
\end{enumerate}
Details on the interpretation of Kac diagrams can be found in the papers
of Kac \cite{Kac94} and Vinberg \cite{Vin76}. Details of the claims made
here will be given in a forthcoming paper by the first author. 
\end{remark}

\begin{example}--- All $\theta$-representations obtained from automorphisms
$\theta$ of order $m=2$ are known to be stable. If they are in addition
locally free, then they are of {\sl maximal rank}, i.e.\ 
$\varphi(m)\dim\gothc=\dim\gothh$, where $\varphi$ denotes Euler's
$\varphi$-function and $\gothh\subseteq \gothg$ is a Cartan subalgebra of $\gothg$.
Thus, for $m=2$, maximality of rank is equivalent to the existence of a Cartan
subalgebra $\gothh$ that is completely contained in $\gothg_1$; see \ \cite[\S 3.1]{Vin76}.

An example of a locally free stable $\theta$-representation which is not of
maximal rank can be obtained as follows: Let $\gothg=\Liesl_m$ and let 
$\theta$ be conjugation by $a=\diag(1,\xi,\ldots\xi^{m-1})$ for some primitive
$m$-th root of unity. Then $G_0\subseteq\LieSl_m$ is the diagonal torus and
$\gothg_1$ is the span of all elementary matrices $E_{i,i+1}$, for $i=1,\ldots,m-1$,
and $E_{m,1}$. If $m$ is not prime then $(G_0,\gothg_1)$ is not of maximal rank.
\end{example}

\begin{example}\label{ex binary tetrahedral}--- The motivating example for this paper is the representation
$(\LieSl_3,S^3\IC^3)$. This is a stable locally free $\theta$-representation 
for an automorphism of order $m=3$ on $\Lieso_8$. If $S^3\IC^3$ is viewed as the
space of plane cubic curves, a Cartan subspace is provided by the Hesse
pencil $\langle x^3+y^3+z^3, xyz\rangle$. The corresponding Weyl group is 
the binary tetrahedral group. The quotient $\gothc\oplus\gothc\dual/W$ is among
the small list of finite symplectic quotients that do admit a symplectic
resolution. Explicit resolutions were given in \cite{LS12}.
\end{example}

%%%%%%%%%%%%%%%%%%%%%%%%%%%%%%%%%%%%%%%%%%%%%%%%%%%%%%%%%%%%%%%%%%%%%%%%%%%%%%%
%%%%%%%%%%%%%%%%%%%%%%%%%%%%%%%%%%%%%%%%%%%%%%%%%%%%%%%%%%%%%%%%%%%%%%%%%%%%%%%
\section{Slices}

In the proof of our main theorem we will need an induction argument
that is based on the passage from a polar representation $(G,V)$ to 
the slice representation $(G_x,N)$ of a semisimple element $x\in\gothc$
in the Cartan subspace. The reduction argument splits into two parts:
A Luna-type slice theorem for symplectic reductions, and a formal 
Darboux theorem. We treat the slice theorem first. For the special
case of a linear action of the group of units in a semisimple algebra, the theorem is due to Crawley-Boevey \cite[\S~4]{CB03}; the general case is due to 
Jung \cite{Jun09}. As Jung's paper is unpublished we give a simplified proof as follows.

Let $(X,\omega)$ be a smooth affine symplectic variety with a Hamiltonian
action by a reductive group $G$, i.e.\ an action that preserves the
symplectic structure and admits a moment map $\mu:X\to \gothg^*$. Let
$x\in \mu^{-1}(0)$ be a point with closed orbit and therefore reductive
stabiliser subgroup $H=G_x$. The defining property of the moment map,
 $d\mu_x(\xi)(A)=\omega(\xi,Ax)$ for all $A\in \gothg$ and $\xi
\in T_xX$, directly implies that $\Imag(d\mu_x)=\gothh^\perp=(\gothg/\gothh)^*$
and $\Kern(d\mu_x)=(\gothg x)^\perp\subseteq T_xX$.

Choose a $G$-equivariant closed embedding $X\subseteq V$ into a linear 
representation of $G$ and $H$-equivariant splittings $V=\gothg x\oplus V'$ and 
$\gothg^*=\gothh^\perp\oplus \gothh^*$. The second splitting yields a 
decomposition $\mu=\mu^{\perp}\oplus \bar \mu$ into two components. The fact 
that $\Imag(d\mu_x)=\gothh^\perp$ implies that the component $\mu^\perp:X\to 
\gothh^\perp$ is smooth at $x$ and hence $Y:=(\mu^\perp)^{-1}(0)=
\mu^{-1}(\gothh^*)$ is smooth at $x$ of dimension $\dim_xY=\dim X-\dim\gothg x$.
In fact, $T_xY=\Kern(d\mu_x)$. Now $S_X:=(x+V')\cap X$ and $S_Y:=(x+V')\cap Y$
are transverse slices at $x$ to the orbit of $x$ in $X$ and $Y$, respectively, 
and the natural projection $T_xS_Y\to (\gothg x)^\perp/\gothg x$ is an
isomorphism. In particular, there is an $H$-stable open affine neighbourhood
$U$ of $x$ in $S_Y$ such that $U$ is smooth and the restriction of $\omega$ to 
$U$ is symplectic. Moreover, the composite map $\mu':S_Y\to X\xra{\;\mu\;}
\gothg^*\to\gothh^*$ is a moment map for the action of $H$ on $S_Y$. By 
construction, $\mu^{-1}(0)\cap S_X=\mu'{}^{-1}(0)$. Hence the inclusion induces 
a natural morphism
$$U\SPR H\subseteq (\mu^{-1}(0)\cap S_X)\GIT H \to \mu^{-1}(0)\GIT G=X\SPR G,$$
and Luna's slice theorem  implies the following result.

\begin{theorem} {\em ('Symplectic Slice Theorem')} --- The morphism
$U\SPR H\rightarrow X\SPR G$
is \'etale at $[x]$.
\end{theorem}

It remains to compare the symplectic reductions of $U$ and its tangent 
space $T_xU=\gothg x^\perp/\gothg x$ at $x$, endowed with the linearised 
action of $H$. 

Let $\gothm\subseteq \widehat\ko_{S_Y,x}$ denote the maximal ideal of the 
completion of the local ring of $S_Y$ at $x$. Any choice of regular 
parameters $\xi_1,\ldots,\xi_n\in\gothm$ induces an isomorphism 
$u:\widehat\ko_{\gothg x^\perp/\gothg x, 0}\to R$ that is characterised 
by $z_i=(\xi_i\,\mod\,\gothm^2)\mapsto \xi_i$, $i=1,\ldots,n$. Let 
$\omega$ denote the symplectic form on $S_Y$ and $\omega^{(0)}$ its 
value at $x$, i.e.\ the corresponding constant symplectic form on $T_x S_Y$.

\begin{theorem} {\em ('Formal Darboux Theorem')} --- There exist regular
parameters $\xi_1,\ldots,\xi_n\in \widehat\ko_{S_Y,x}$, $n=\dim_{x}S_Y$, such
that the induced isomorphism $u:\widehat\ko_{\gothg x^\perp/\gothg x, 0}
\to \widehat\ko_{S_Y,x}$ satisfies $u^*\omega=\omega^{(0)}$.
\end{theorem}

\begin{proof} Let $z_1,\ldots,z_n\in\gothm$ be any set of regular parameters.
Since $\omega$ is closed, the Poincar\'{e} lemma allows one to find $1$-forms
$\psi^{(m)}$, $m\geq 2$, with homogeneous coefficients of degree $m$, 
such that $\omega$ may be written as 
$$\omega=\sum_{\alpha,\beta}\Omega_{\alpha\beta}dz_\alpha
\wedge dz_\beta + d\psi^{(2)}+d\psi^{(3)}+\ldots$$
for some non-degenerate skew-symmetric matrix $\Omega$. We need to find power 
series 
$$\xi_i:=z_i+\sum_{m\geq 2} \xi_i^{(m)},$$
where each $\xi_i^{(m)}$ is a homogeneous polynomial in $z_1,\ldots,z_n$ of degree 
$m$, such that
\begin{equation}\label{eq:DarbouxCondition}
\omega=\sum_{\alpha,\beta}\Omega_{\alpha\beta}d\xi_\alpha\wedge d\xi_\beta.
\end{equation}
This equation is certainly satisfied if the $\xi_\beta^{(m)}$ satisfy the
relations
\begin{equation}\label{eq:DarbouxRecursion1}
\psi^{(m)}-\sum_{i=1}^{m-2} \sum_{\alpha,\beta}\Omega_{\alpha\beta}
\xi_\alpha^{(m-i)} d\xi_{\beta}^{(i+1)}=2\sum_{p,q} \Omega_{pq}\xi_p^{(m)} 
dz_q
\end{equation}
for $m\geq 2$. Let $\iota_{\partial_q}$ denote the partial evaluation 
along the vector field $\partial_q$ satisfying $\partial_q(z_p)=\delta_{pq}$. 
Then \eqref{eq:DarbouxRecursion1} may be recast into the form
\begin{equation}
\xi_p^{(m)}=\frac 12 \sum_q\Omega^{-1}_{pq}\iota_{\partial_q}\left(\psi^{(m)}
-\sum_{i=1}^{m-2} \sum_{\alpha,\beta}\Omega_{\alpha\beta}
\xi_\alpha^{(m-i)} d\xi_{\beta}^{(i+1)}\right)
\end{equation}
which provides a means of computing the $\xi^{(m)}$, $m\geq 2$, recursively.
\end{proof}

\begin{lemma}\label{lem:inheritanceofproperties}
--- Let $(G,V)$ be a representation of a reductive algebraic group, let $w\in V$ be a semisimple element, and let $(G_w,N)$ denote the slice representation associated to $w$. If $(G,V)$ satisfies any of the properties of being
\begin{enumerate}
\item visible,
\item locally free, or
\item stable,
\end{enumerate}
then the same property holds for $(G_w,N)$.
\end{lemma}

Note that slice representations are defined for an arbitrary representation thanks to \cite[\S~6.5]{PV89} and that these coincide with slice representations of \cite[\S~2]{DK85}.

\begin{proof} 

(1) This is contained in the Encyclopedia article \cite[Theorem 8.2]{PV89} by Popov and Vinberg.

% Let $Y\subseteq N$ be a $G_w$-stable irreducible subset of the 
% nilpotent locus, so $0\in \overline Y$, and consider the dominant morphism
% $\gamma:G\times \overline Y\to \overline{G.(w+Y)}$, $(g,y)\mapsto g.(w+y)$.
% Let $\gothl\subseteq\gothg$ be a $G_w$-stable complement to $\gothg_w$. Then
% $\gothl w=\gothg w$ is a complement to $N$ in $V$ of dimension $\dim\gothl$.
% Hence, there exists a neighbourhood $U$ of $0$ in $\overline Y$ such that 
% $\dim \gothl(w+y)=\dim\gothl$ for all $y\in U$ and $N+\gothl(w+y)=V$. Therefore, the image
% of the differential $d\gamma_{e,y}$ equals $\gothg.(w+y)+T_y\overline Y=
% \gothl.(w+y)+\gothg_w(w+y)+T_y\overline Y$. By the $G_w$-stability of $Y$,
% one has $\gothg_w(w+y)\subseteq \gothg_wy\subseteq T_y\overline Y\subseteq N$. Hence,
% for any smooth point $y$ of $Y$ that belongs to $U$, one has
% $\dim(\Imag d\gamma_{e,y})=\dim Y+\dim\gothl=\dim Y+\dim G.w$. This proves:
% $$\dim G.(w+Y) =\dim Y+\dim G.w.$$
% Applied to the orbit $G_wn$ of a $G_w$-nilpotent element $n\in N$ this shows:
% $$\dim G.(w+n)= \dim G.(w+G_w n)=\dim G.w+\dim G_w.n.$$ 
% So if $Y$ is a sheet of $G_w$-orbits in the nilpotent locus of 
% $N$, then 
% $$\dim G.(w+Y)= \dim Y+\dim G.w=\dim G.(w+n)+(\dim Y-\dim G_w.n)$$
% for all $n\in Y$. But the visibility of $(G,V)$ implies that there are only 
% finitely many orbits in the fibre $\pi^{-1}(\pi(w))$, so that 
% $\dim G.(w+Y)=\dim G.(w+n)$. This implies $\dim Y=\dim G_w.n$, i.e.\ 
% $Y$ consists of a single orbit. Hence, $(G_w,N)$ is visible.

(2) The image of the differential of the morphism $\gamma:G\times N\to V$, 
$(g,n)\mapsto gn$, at $(e,w)$ equals $\gothg.w+N=V$. Consequently, $\gamma$ is dominant.
Thus, for a general element $n$ in $N$, the stabiliser subgroup in $G$ is finite,
and a fortiori its stabiliser in $G_w$ is finite. Hence $(G_w,N)$ is locally free.

(3) A representation is stable if and only if the set of semisimple elements contains a dense open subset. We deduce the existence of a non-empty (and hence dense) open subset $U\subseteq N$ consisting of closed orbits for the $G_w$ action
from Luna's fundamental lemma \cite[II.2]{Lu73} and \cite[I.3 Lemme]{Lu73}.
\end{proof}

\begin{remark} ---
We will not use this fact in the following but it is worth noting that if the morphism 
$r:\gothc\oplus \gothc^*/W\to (V\oplus V^*\SPR G)_{\red}$ is bijective, then so is the induced morphism 
$\gothc\oplus \gothc^*/W_w\to (N\oplus N^*\SPR G_w)_{\red}$, where $W_w=N_{G_w}(\gothc)/Z_{G_w}(\gothc)$.
\end{remark}

\begin{proposition}\label{prop:EqualCompletedRings}
--- Let $(G,V)$ be a polar representation with Cartan subspace 
$\gothc$. Let $N$ denote a $G_x$-stable complement to $\gothg x$ for a point 
$x\in \gothc$. Then $\widehat\ko_{N\oplus N^*\SPR G_x, [0,0]} \isom 
\widehat\ko_{V\oplus V^*\SPR G, [(x,0)]}.$
\end{proposition}

\begin{proof} The stabiliser of the point $(x,0)\in V\oplus V^*$ is $G_x$, and
$N\oplus V^*$ is a $G_x$-stable slice to the orbit. The annihilator of $\gothg x$
with respect to the symplectic structure is $(\gothg x)^\perp=\gothg x\oplus N^*$, 
so that $N\oplus V^*)\cap (\gothg x)^{\perp}=N\oplus N^*$. Hence the assertion
follows from the symplectic slice theorem and the formal Darboux theorem.
\end{proof}

%%%%%%%%%%%%%%%%%%%%%%%%%%%%%%%%%%%%%%%%%%%%%%%%%%%%%%%%%%%%%%%%%%%%%%%%%%%%%%%
%%%%%%%%%%%%%%%%%%%%%%%%%%%%%%%%%%%%%%%%%%%%%%%%%%%%%%%%%%%%%%%%%%%%%%%%%%%%%%%
\section{Proof of Theorem \ref{thm:maintheorem}}

We will now prove Theorem \ref{thm:maintheorem} from the introduction. Let $(G,V)$ denote a stable locally free and visible polar representation.

Pan\-yushev \cite[Thm.~3.2]{Pan94} showed that for locally free, stable, visible
polar representations the null-fibre $\mu^{-1}(0)$ is a reduced and irreducible 
complete intersection of expected dimension $2\dim V-\dim G=\dim V+\dim \gothc$.
In particular, the symplectic reduction is a variety, i.e.\ reduced and irreducible.
According to Propositions \ref{prop:finiteinjective} and \ref{prop:ontoC0} 
and Lemma \ref{lem:IrreducibilityCriterion}, the morphism 
$r:\gothc\oplus\gothc\dual/W\to V\oplus V^*\SPR G$ is the normalisation
of the symplectic reduction. Hence showing that $r$ is an isomorphism is equivalent 
to proving that $V\oplus V^*\SPR G$ is normal. (Unfortunately, it is not true in general that the null-fibre itself is normal; see Remark \ref{rk_normal_muz}.)

We will argue by induction on the {\sl rank}
$$\rk(G,V)=\dim \gothc -\dim (\gothc\cap V^G).$$ 
Because of the stability assumption, $\rk(G,V)=0$ is equivalent to $V$ being
a trivial representation.
Note that we can in fact always assume that $V$ contains no trivial summand;
for otherwise we may decompose $V$ as $V=\IC^\ell\oplus V_0$ with a 
representation $V_0$ that is again stable, locally free, visible, and polar. Clearly,
the corresponding Cartan subspaces are related by $\gothc=\IC^\ell\oplus \gothc_0$,
and the symplectic reductions by $V\oplus V^*\SPR G= \IC^{2\ell}\times
(V_0\oplus V_0^*)\SPR G$.

If $\rk(G,V)=1$, the assertion follows from Proposition \ref{prop:rank1case}.

Assume now that $\rk(G,V)\geq 2$, that $V$ contains no trivial summand, and
that the assertion has been proved for all stable, locally free, visible polar
representations of lower rank.

Let $(G_w,N)$ denote the slice representation for some point $w\in \gothc\setminus\{0\}$.
According to Lemma \ref{lem:inheritanceofproperties}, the slice representation
is again visible, locally free and stable. Since $w\in \gothc\cap N^{G_w}$, one has
$$\rk(G_w,N)=\dim\gothc-\dim(\gothc\cap N^{G_w})<\dim \gothc=\rk(G,V).$$
By our induction hypothesis, $N\oplus N^*\SPR G_w$ is normal. By Zariski's theorem,
\cite[Vol.~II, p.\ 320]{ZS58} the completion of the local ring
$\ko_{N\oplus N^*\SPR G_w,[0,0]}$ is normal. By Proposition \ref{prop:EqualCompletedRings}
the completion of $\ko_{V\oplus V^*\SPR G,[w,0]}$ is normal. Hence the symplectic
reduction is normal in a neighbourhood $B$ of the image of $(\gothc\setminus\{0\})
\times \{0\}$ and, for symmetry reasons, also in a neighbourhood $B'$ of the image
of $\{0\}\times(\gothc\dual\setminus\{0\})$. 

Consider the action of $\IC^*\times\IC^*$ on $V\oplus V^*$ given by $(t_1,t_2).(x,\varphi)
=(t_1x,t_2\varphi)$. It commutes with the $G$-action and preserves the null-fibre.
Therefore, it descends to an action on the symplectic reduction. Clearly, the orbit
of any point other than the origin in the symplectic reduction meets $B$
or $B'$. This implies that $(V\oplus V^*\SPR G)\setminus \{[0,0]\}$ is normal.

Since the rank of $(G,V)$ is at least $2$, the codimension of the origin in 
the symplectic reduction is at least $4$, so that the symplectic reduction is
regular in codimension 1. Using Serre's criterion for normality, it suffices
to show that the symplectic reduction satisfies property $(S_2)$. According 
to a theorem of Crawley-Boevey \cite[Thm.~7.1]{CB03}, it suffices to verify the
following assumptions in order to conclude that $V\oplus V^*\SPR G$ has property
$(S_2)$: (i) $\mu^{-1}(0)$ has property $(S_2)$; (ii) $(V\oplus V^*\SPR G)\setminus\{0\}$
has property $(S_2)$; (iii) $Z:=q^{-1}(0)\cap \mu^{-1}(0)$ has codimension at least $2$ 
in $\mu^{-1}(0)$, where $q:V\oplus V^*\to V\oplus V^*\GIT G$ denotes the quotient
map. 

As we have already seen that (i) and (ii) hold, it suffices to bound the codimension
of $Z:=q^{-1}(0)\cap \mu^{-1}(0)$. Any point in $Z$ is of the form $(n,\varphi)$,
where $n\in V$ is nilpotent, i.e.\ $0\in \overline{Gn}$ and $\varphi\in 
(\gothg n)^\perp$. Since $V$ is visible, there are only finitely many nilpotent 
orbits $Gn_1, \ldots, Gn_s$ in $V$, so 
$Z\subseteq \bigcup_{i=1}^s G.(\{n_i\}\times (\gothg n_i)^{\perp})$ and 
$\dim Z\leq \max\{ \dim \gothg n_i+(\dim V-\dim \gothg n_i)\}=\dim V$. As 
$\dim \mu^{-1}(0)=\dim V +\dim \gothc\geq \dim V +2$, the theorem is proved.\qed

\begin{remark}
The proof above can be adapted to show that for a visible polar representation $(G,V)$, if the morphism $r:\gothc\oplus \gothc^*/W\to (V\oplus V^*\SPR G)_{\red}$ is bijective, then $(V\oplus V^*\SPR G)_{\red}$ is smooth in codimension $1$. In particular, by Proposition \ref{prop:theta}, the conjecture holds for an arbitrary $\theta$-representation if and only if $(V\oplus V^*\SPR G)_{\red}$ satisfies Serre's condition $(S_2)$. 
\end{remark}

%%%%%%%%%%%%%%%%%%%%%%%%%%%%%%%%%%%%%%%%%%%%%%%%%%%%%%%%%%%%%%%%%%%%%%%%%%%%%%%
%%%%%%%%%%%%%%%%%%%%%%%%%%%%%%%%%%%%%%%%%%%%%%%%%%%%%%%%%%%%%%%%%%%%%%%%%%%%%%%
\section{Examples and Counterexamples} \label{examples}

We first discuss several cases where the conjecture holds but which are not
covered by general results. The tools are representation theory and classical
invariant theory.

\begin{example}--- The conjecture holds for the representations
$(\LieSl_n, S^2\IC^n)$ and $(\LieSl_{n}, \Lambda^2\IC^{n})$. These are polar
but not $\theta$-representations. The cases are very similar, and we will
give details only for $(\LieSl_{n},\Lambda^2\IC^{n})$ with $n$ even.

Let $n=2m$ and identify both $V=\Lambda^2\IC^n$ and its dual $V^*$ with the
space of skew-symmetric matrices, so that the action of $\LieSl_n$ on 
$V\oplus V^*$ is given by $g.(A,B)=(gAg^t, (g^t)^{-1}Bg^{-1})$. Through this
identification, the moment map becomes $\mu(A,B)=AB-\frac1n\tr(AB)I_n$, so 
that 
$$\mu^{-1}(0)=\{(A,B)\;|\; AB=\frac1n \tr(AB)I_n\}.$$
On the other hand, we see from Schwarz's Table 1a in \cite{Sch78} that the
invariant algebra is generated by the Pfaffians $X=\pf(A)$ and $Y=\pf(B)$ and
the traces $Z_i=\frac1n \tr((AB)^i)$ for $1\leq i\leq m-1$. In the coordinate
ring $\IC[\mu^{-1}(0)]$ of the null-fibre, the restrictions of these invariants 
satisfy the relations $Z_i= Z_1^i$ and $XY=Z_1^m$ and no more. It follows that
$V\oplus V^*\SPR \LieSl_n$ is a normal surface singularity of type $A_{m-1}$.
So the conjecture holds.
\end{example}

\begin{example}--- The conjecture holds for the standard representations
$(G_2,\IC^7)$ and $(F_4,\IC^{26})$ and the spin representation $(\LieSpin_7,\IC^8)$.
In all cases, the null-fibre of the moment map is non-reduced, but the 
symplectic reduction is normal. Again, the cases are very similar, and we give 
details only for the height-dimensional spin representation $V$. As $V$ has a 
$\LieSpin_7$-invariant non-degenerate quadratic form $q$, the representation is 
self-dual, $V\isom V^*$. According to Schwarz \cite[Thm.~4.3]{Sch07}, the 
invariant algebra $\IC[V\oplus V]^{\LieSpin_7}$ is generated by the invariants
$A(v+v')=q(v)$ and $B(v+v')=q(v,v')$, the polarisation of $v$, and $C(v+v')=q(v')$.
The quadratic part of the coordinate ring $\IC[V\oplus V]$ equals 
$S^2V\oplus V\tensor V\oplus S^2V$, and the summand $V\tensor V$ that contains
the generators of the ideal $I$ of the null-fibre $\mu^{-1}(0)$ further decomposes
into 
$$V\tensor V=S^2V\oplus \Lambda^2V\isom (V_{(0,0,0)}\oplus V_{(0,0,2)})
\oplus (V_{(0,1,0)}\oplus V_{(1,0,0)}),$$
where $V_{(k_1,k_2,k_3)}$ denotes the irreducible representation with highest weight
$k_1\varpi_1+k_2\varpi_2+k_3\varpi_3$. One can check that $I$ is in fact generated
by $V_{(0,1,0)}$, and that $J^2\subseteq I\subsetneq J$, where $J$ is the ideal
generated by $V_{(0,1,0)}\oplus V_{(0,0,2)}$. In particular, $\mu^{-1}(0)$ is
non-reduced. A direct computation with \textsc{Singular} \cite{DGPS}
or \textsc{Macaulay2} \cite{Mac2} shows that $\IC[\mu^{-1}(0)]^{\LieSpin_7}=
\IC[A,B,C]/(AC-B^2)$ is a normal ring. 
\end{example}

\begin{example}\label{ex:Spin9} --- The conjecture holds for the spin representation
$(\LieSpin_9, V=\IC^{16})$. In this example, even the symplectic reduction
is non-reduced, but $(V\oplus V^*\SPR \LieSpin_9)_{\red}$ is normal. As
before, $V_{(k_1,k_2,k_3,k_4)}$ will denote the irreducible representation
of highest weight $k_1\varpi_1+k_2\varpi_2+k_3\varpi_3+k_4\varpi_4$. In 
particular, the spin representation is $V=V_{(0,0,0,1)}$. As for $\LieSpin_7$,
there is an invariant quadratic form $q$ on $V$ that allows the identification of $V$
and $V^*$, and whose polarisations provide invariants of bidegree $(2,0)$,
$(1,1)$ and $(0,2)$ in $\IC[V\oplus V]^{\LieSpin_9}$. It is well known that
the null-fibre $\mu^{-1}(0)$ has two irreducible components (cf.\ 
\cite[\S~4]{Pan04}). Let $I\subseteq \IC[V\oplus V]$ denote the vanishing 
ideal of the null-fibre. It is generated by the 36-dimensional summand
$V_{(0,1,0,0)} (\isom \Lieso_9)$ in the decomposition
$$\IC[V\oplus V]_2\supseteq V\tensor V\isom (V_{(0,0,0,0)}\oplus V_{(0,0,0,2)}\oplus
V_{(1,0,0,0)})\oplus ( V_{(0,0,1,0)}\oplus V_{0,1,0,0}).$$
Moreover, according to Schwarz \cite[Table 3a]{Sch78}, the invariant ring
$\IC[V\oplus V]^{\LieSpin_9}$ is generated by the aforementioned invariants
$A$, $B$ and $C$ together with and an invariant $D$ of bidegree $(2,2)$. The explicit 
description of $D$ is more involved. It arises as the unique invariant quadratic
form on $V_{0,1,0,0}\isom \Lambda^3\IC^9$. 

Using \textsc{Singular} or \textsc{Macaulay2}, one can check that the invariant
$P=AC-B^2$ satisfies $P\not\in I$ but $P^2\in I$, so that $V\oplus V^*\SPR
\LieSpin_9$ is not reduced.

From Tevelev's Theorem \cite{Tev00} we already know that the conjecture holds
for the spin representation of $\LieSpin_9$ since it can be realized as the 
isotropy representation of the symmetric space $F_4/\LieSpin_9$, i.e.\ a 
$\theta$-representation of order $m=2$. Explicitly, this can be seen as follows.
First, one verifies (for instance by using the software \textsc{LiE} \cite{LCL})
that $V\tensor V\tensor V_{(0,1,0,0)}$ contains a unique copy of the trivial
representation, and thus the bidegree-$(2,2)$ component $I_{(2,2)}$ of the
ideal $I$ contains at most one invariant. Using \textsc{Macaulay2}, one computes
the rank of the multiplication map $\kappa:I_{(1,1)}\tensor I_{(1,1)}\to 
I_{(2,2)}$ to be 666. The only dimension match for a submodule in $V\tensor V
\tensor V_{(0,1,0,0)}$ is $\Imag(\kappa)=V_{(0,0,0,0)}\oplus V_{(0,0,0,2)}\oplus 
V_{(0,2,0,0)}\oplus V_{(2,0,0,0)}$. In particular, $I_{(2,2)}^{\LieSpin_9}=
\IC Q$ for some linear combination $Q=\alpha D\oplus \beta AC+\gamma B^2$ with 
$\alpha,\beta,\gamma\in\IC$. A \textsc{Macaulay2} calculation gives 
$I\cap \IC[A,B,C]=(A,B,C)(AC-B^2)$. Hence the coefficient $\alpha$ must be 
non-zero. It follows that 
$I^{\LieSpin_9}=(Q,A(AC-B^2),B(AC-B^2),C(AC-B^2))$ and $\sqrt{I^{\LieSpin_9}}
=(Q,AC-B^2)$, whence 
$$\IC[\mu^{-1}(0)_{\red}]^{\LieSpin_9}=\IC[A,B,C,D]/(Q,AC-B^2)=\IC[A,B,C]/(AC-B^2),$$
a normal ring.
\end{example}

\begin{remark} --- 
One of the features of the representation $\LieSpin_9$ is that $\mu^{-1}(0)$ is reducible. Hence one might ask whether for every isotropy representation $(G,V)$ of a symmetric space whose commuting scheme $\mu^{-1}(0)$ is reducible we always have that $V\oplus V^*\SPR G$ is non-reduced. The answer is actually negative. For instance, if $$(G,V)=(\LieGl(V_1) \times \LieGl(V_2), \Hom(V_1,V_2) \times \Hom(V_2,V_1))$$ 
with $\dim V_1=1<\dim V_2$, then $(G,V)$ is the isotropy representation of a symmetric space \cite[Ex.~4.3]{PY07}, $\mu^{-1}(0)$ is reduced and has three irreducible components \cite[\S~2.2.1]{Ter12}, and $V\oplus V^*\SPR G$ is a normal variety \cite[\S~2.2.2]{Ter12}. 
\end{remark}

The following proposition indicates that the visibility assumption in Conjecture \ref{conj}
should not be dropped.

\begin{proposition}\label{prop:nonvisible} --- Let $(G,V)$ be a polar 
representation of a reductive algebraic group $G$ such that $V$ contains a non-trivial 
representation $E$ with multiplicity greater than or equal to $2$. Then $V$ is non-visible and 
$\gothc\oplus\gothc\dual/W$ maps to a proper closed subset of 
$(V\oplus V^*\SPR G)_{\red}$. 
\end{proposition}

\begin{proof} Consider a decomposition $V=E\oplus E\oplus V'$. Since $E$ is a 
non-trivial representation, the general fibre, and hence every fibre, of the 
quotient map $E\to E\GIT G$ has positive dimension. In particular, there exists
a nilpotent element $y\in E \setminus \{0\}$, i.e.\ an element with $0\in\overline{Gy}$. For 
different values of $t\in\IC$, the pairs $(y,t y)$ belong to different 
nilpotent $G$-orbits in $E\oplus E$. Thus, the number of orbits in the null-fibre
of $V\to V\GIT G$ is infinite, and $V$ is not visible.

By the Hilbert-Mumford criterion \cite[Ch.~2, \S 1]{MF82}, there exists
a one-parameter subgroup $\lambda:\IG_m\to G$ such that $\lim_{t\to 0} \lambda(t)y=0$.
If $E=\bigoplus_{n\in\IZ} E_\lambda(n)$ and $E^*=\bigoplus_{n\in\IZ} E_\lambda(n)^*$
are the associated weight space decompositions of $E$ and $E^*$ one has $y\in
\bigoplus_{n>0}E_\lambda(n)$. We may choose $\psi\in \bigoplus_{n>0}E^*_\lambda(n)$
with $\psi(y)\neq 0$. Then $\psi$ is nilpotent as well. Consider the elements
$w=(y,y,0)\in V$ and $\varphi=(\psi,-\psi,0)\in V^*$. By construction, 
$w+\varphi \in \mu^{-1}(0)$.
The invariant $V\oplus V^*\ni (a,b,c)+(\alpha,\beta,\gamma)\mapsto \alpha(a)$ takes
the non-zero value $\psi(y)$ at the point $w+\varphi$, so that $w+\varphi$ cannot 
be nilpotent. Let $w'+\varphi'$ be a representative of the closed orbit in 
$\overline{G(w+\varphi)}$. Then $w'$ and $\varphi'$ must have the forms
$w'=(y',y',0)$ and $\varphi'=(\psi',\psi',0)$ with nilpotent elements $y'$ and $\psi'$
satisfying $\psi'(y')=\psi(y)\neq 0$. In particular, the orbit $w'+\varphi'$ does not
meet $\gothc\oplus \gothc\dual$. Hence $C_0\GIT G$ is a proper closed subset of
the symplectic reduction.
\end{proof}

\begin{example}\label{ex visible stable locally free not theta}
Here we construct a family of examples of visible stable locally free polar representations which are not $\theta$-representations.

Fix $n\geq 1$. Take an $n$-dimensional torus $G=\IG_m^n$ acting on $V=\IC^{n+1}$ with weights $(a_1,a_2,\ldots,a_n, -\sum_{i=1}^na_i)$, where $a_1,\ldots,a_n$ are linearly independent. It is easily checked that the representation $(G,V)$ is visible, stable and locally free. As the invariant ring $k[V]^G$ is generated by a single polynomial (the product of the $n+1$ coordinate vectors), it is automatically polar (every non-zero semisimple element spans a Cartan space).
Using Vinberg's theory \cite{Vin76}, one can check that if $(G,V)$ is a $\theta$-representation, then there is a simple graded Lie algebra $(\tilde\gothg,\theta)$ such that $G=\tilde{G}_0$ and $V=\tilde{\gothg}_1$. From the classification of $\theta$-representations by Kac diagrams one deduces that there are only finitely many possibilities for $(\tilde{G}_0,\tilde{\gothg}_1)$ up to isomorphism but there is certainly an infinite number of weights $a_1, \ldots, a_n$ as above. For instance, for $n=1$ the representation $(G,V)$ is a $\theta$-representation if and only if $a_1=\pm 2$.

This family of examples provides a hint that the class of polar representations for a general reductive Lie algebra should be much richer than the class of $\theta$-representations and underlines the importance of working in the general context of polar representations. 
\end{example}

\begin{example}\label{ex nonvisible} --- 
Here are several examples of non-visible polar representations to which Proposition \ref{prop:nonvisible} applies. 
\begin{enumerate}
\item Let $(G,V)=(\LieSp_n, \IC^n \oplus \IC^n)$ with $n\geq 2$ even. Then $V\oplus V^*\SPR G$ is the union of two irreducible normal surfaces $C_0\GIT G$ and $F$; see \cite[\S~3.5]{Ter12} for details. 
\item Let $(G,V)=(\LieSl_2, \IC^2 \oplus \IC^2)$, which is a stable representation with a locally free action. Then one may check that $\mu^{-1}(0)$ is a five-dimensional complete intersection with two irreducible components (meeting in the four-dimensional singular locus) and that $V\oplus V^*\SPR G$ is the union of two normal surfaces $C_0\GIT G$ and $F$ with a single isolated $A_1$-singularity each and meeting in the singular point; see \cite[\S~3]{Bec09} for details. More generally, we can consider $(G,V)=(\LieSl_n, (\IC^n)^{\oplus n})$, but then the general description of $\mu^{-1}(0)$ is more involved. 
\item Let $(G,V)=(\LieSpin_{10}, E \oplus E)$, where $E$ is the spin representation. Then $V\oplus V^*\SPR G$ is of dimension at least $4$ while $C_0\GIT G$ is of dimension $2$.
\end{enumerate}
\end{example}

\begin{remark} --- 
The previous example suggests that a weaker form of the conjecture might be true for arbitrary polar representations provided that we replace ${(V\oplus V^*\SPR G)}_{\mathrm{red}}$ by $C_0\GIT G$ in the statement.
\end{remark}

\noindent \textbf{Acknowledgments.}
We would like to thank Dmitry Kaledin, Peter Littelmann, Dimitri Panyushev, and Christoph Sorger for interesting discussions related to this work.
We also thank the referee for useful remarks, in particular for pointing out that some previously made assumptions in Lemma \ref{lem:inheritanceofproperties} were unnecessary.
The second author gratefully acknowledges the support of the DFG through the research grants Le 3093/2-1 and Le 3093/3-1.
The third author was supported by SFB Transregio 45 ”Periods, Moduli Spaces and Arithmetic of Algebraic Varieties”.
The fourth author is grateful to the Max-Planck-Institut f\"{u}r Mathematik of Bonn for the warm hospitality and support provided during the writing of this paper.
Part of the genesis of this paper occurred during the half-semester programme 
\emph{Algebraic Groups and Representations} supported by Labex MILYON/ANR-10-LABX-0070.

%%%%%%%%%%%%%%%%%%%%%%%%%%%%%%%%%%%%%%%%%%%%%%%%%%%%%%%%%%%%%%%%%%%%%%%%%%%%%%%
%%%%%%%%%%%%%%%%%%%%%%%%%%%%%%%%%%%%%%%%%%%%%%%%%%%%%%%%%%%%%%%%%%%%%%%%%%%%%%%

\bibliographystyle{alpha}

\end{document}